\documentclass[12pt,reqno]{amsart}

\usepackage{amssymb}
\usepackage{ifthen}
\usepackage{verbatim}

\renewcommand{\phi}{\varphi}

\newcommand{\Z}{{\mathbb Z}}

\newcommand{\N}{{\mathbb N}}

\newcounter{step}

\theoremstyle{plain}
\newtheorem{lemma}{Lemma}
\newtheorem{theorem}{Theorem}

\newtheorem{corollary}{Corollary}

\theoremstyle{definition}

\newcommand{\reft}[1]{~\ref{t:#1}}

\newcommand{\refe}[1]{~\eqref{e:#1}}

\begin{document}

\title[Inverse problems in group theory]
{Inverse problems in Additive Number Theory\\ and in Non-Abelian Group
Theory}

\author{G. A. Freiman}
\address{Gregory A. Freiman\vspace{-0.25cm}}
\address{School of Mathematical Sciences, Tel Aviv University,
Tel Aviv 69978, Israel.}
\email{~grisha@post.tau.ac.il}

\author{M. Herzog}
\address{Marcel Herzog\vspace{-0.25cm}}
\address{School of Mathematical Sciences, Tel Aviv University,
Tel Aviv 69978, Israel.}
\email{~herzogm@post.tau.ac.il}

\author{P. Longobardi}
\address{P. Longobardi\vspace{-0.25cm}}
\address{Dipartimento di Matematica e Informatica, Universita di Salerno, 84084 Fisciano (Salerno), Italy}
\email{~plongobardi@unisa.it}

\author{M. Maj}
\address{M. Maj\vspace{-0.25cm}}
\address{Dipartimento di Matematica e Informatica, Universita di Salerno, 84084 Fisciano (Salerno), Italy}
\email{~mmaj@unisa.it}

\author{Y. V. Stanchescu}
\address{Yonutz V. Stanchescu\vspace{-0.25cm}}
\address{The Open University of Israel, Raanana 43107, Israel
 \newline \normalfont{${\quad}\quad$ and}\vspace{-0.25cm}}
\address{Afeka Academic College, Tel Aviv 69107, Israel.}
\email{ionut@openu.ac.il ~and ~yonis@afeka.ac.il}
\thanks{{\it Keywords}: Inverse additive number theory; Minkowski sums of dilates;
Baumslag-Solitar groups and monoids, inverse non-abelian group theory, extended inverse problems}
\thanks{{\it Mathematics Subject Classification} 2010: Primary 11P70; Secondary 11B30, 20F05, 20F99, 11B13, 05E99.}



\maketitle

\section{Introduction}
The aim of this paper is threefold:

\noindent
a) Finding new direct and inverse results in the additive number theory concerning
Minkowski sums of dilates.

\noindent
b) Finding a connection between the above results and some direct and inverse problems
in the theory of Baumslag-Solitar (non-abelian) groups.

\noindent
c) Solving certain inverse problems in  Baumslag-Solitar groups or monoids, assuming
appropriate small doubling properties. 
\medskip

We start with our first topic (a), concerning the additive number theory.
In this paper $\Z$ denotes the rational integers, $\N$
denotes the
{\bf non-negative} elements of $\Z$ and the size of a finite set $A$ will be
denoted by $|A|$.
Subsets of $\Z$ of the form
$$
r \ast A = \{rx : x\in A\},
$$
where $r$ is a positive integer and $A$ is a {\bf finite} subsets of $\Z$,
are called {\it $r$-dilates}.
 
Minkowski sums of dilates are defined as follows:
$$
r_1 \ast A +...+ r_s \ast A = \{r_1x_1+...+r_sx_s : x_i \in A,\  1 \le i \le
s\}, 
$$
These sums have been recently studied in different situations by
Nathanson, Bukh,  Cilleruelo, Silva, Vinuesa,  Hamidoune, Serra and Ru{\'e}
(see \cite{b:nat}, \cite{b:bukh}, \cite{b:cill-2}, \cite{b:cill-1},
\cite{b:ham}). In particular, they examined
sums of two dilates of the form
$$
A+ r \ast A=\{a+rb\mid \ a,b\in A\}
$$
and solved various {\it direct} and {\it inverse} problems concerning their sizes.

For example, it was shown in \cite{b:cill-2} that
$$|A+2 \ast A|\geq 3|A|-2,$$
which represents a {\it direct} result. Moreover, they solved the following
{\it inverse} problem:
what is the {\it structure} of the set $A$ if $$|A+2 \ast A|= 3|A|-2?$$ 
Their answer was that
in such case $A$ must be an arithmetic progression. 

Inverse problems of this type,
where the exact bound is assumed,
will be called {\it ordinary inverse problems}. The term {\it
extended inverse problem} will refer to inverse problems in which a small 
diversion from the exact bound is allowed, still enabling  us to reach a definite
conclusion concerning the {\it structure} of $A$. 

As an example of an 
extended inverse problem, consider the following question: what is the structure of the set $A$ 
if 
$$|A+2 \ast A|<  4|A|-4?$$
Our answer to this question is:

{\it {\bf (A).} If $|A+2 \ast A|<  4|A|-4$, then $A$ is a {\bf subset} 
of an arithmentic progression of size $\leq 2|A|-3$.}
(see Theorem \reft{inverse-integer}, Section 3)

\medskip

The above mentioned authors and others studied also
the sums $A+ r \ast A$ for $r\geq 3$. In this direction we proved the following
new (direct) result:
 
{\it {\bf (B).} If $r\ge 3$, then} $|A+ r \ast A|\ge 4|A|-4$.
(see Theorem \reft{A+kA}, Section 5) 

This very useful result yields a {\bf uniform} bound for all sets $A$
and for $r\ge 3$. In the literature, most bounds of this type are asymptotic.   

\medskip

We continue now with the second topic (b), dealing with a connection, noticed by us, between
results concerning sums of dilates and some problems in the theory of Baumslag-Solitar groups. 

If $S$ and $T$ are subsets of a group $G$, their {\it product} is defined
as follows:
$$ST=\{st\mid \ s\in S,\ t\in T\}.$$
In particular, $S^2=\{s_1s_2\mid \ s_1,s_2\in S\}$ and
if $b\in G$, then $bS=\{bs\mid \ s\in S\}$.

For integers $m$ and $n$, the general Baumslag-Solitar group 
$BS(m,n)$ is a group with two generators
$a, b$ and one defining relation $b^{-1}a^mb=a^n$:
$$BS(m,n)=\langle a,b\mid \ a^mb=ba^n\rangle.$$
We shall concentrate on
$$G=BS(1,n)=\langle a,b\mid \ ab=ba^n\rangle.$$

\medskip

Let $S$ be a finite subset of $G$ of size $k_1$ contained in the coset 
$b^r\langle a\rangle$ for some $r\in \N$ and let $T$ be a finite subset of $G$ of size $k_2$
contained in the coset
$b^s\langle a\rangle$ for some $s\in \N$. 
Then 
$$S=\{b^ra^{x_0},b^ra^{x_1},\dots,b^ra^{x_{k_1-1}}\},$$
where $A=\{x_0,x_1,\dots,x_{k_1-1}\}$ is a subset of $\Z$. We introduce
now the notation
$$S=\{b^ra^x: x\in A\}=b^ra^A.$$
Thus $|S|=|A|$.

Similarly, $T=b^sa^B$ for some subset $B=\{y_0,y_1,\dots,y_{k_2-1}\}$ of $\Z$.
Since $ab=ba^n$, it follows that $a^{-1}b=ba^{-n}$ and 
\begin{align}\label{e:r-product}
a^xb^t=b^ta^{n^tx}\qquad \text{for each}\qquad x\in \Z\quad 
\text{and}\quad t\in \N.
\end{align}
In particular,
$$a^xb=ba^{nx}\qquad \text{for each}\qquad x\in \Z.$$
Equation (\ref{e:r-product}) implies that 
$$(b^ra^x)(b^sa^y)=b^r(a^xb^s)a^y=b^r(b^sa^{n^sx})a^y=b^{r+s}a^{n^sx+y}$$
for each $x,y\in \Z$ and for each $r,s\in \N$.
Therefore the {\it product set}
$$ST=\{vw\mid \ v\in S,\ w\in T\}$$
can be written as
\begin{align}\label{e:product}
ST&=\{(b^ra^{x_i})(b^sa^{y_j})\mid\ i\in \{0,1,\dots,k_1-1\},\ j\in \{0,1,\dots,k_2-1\}\}\\
&=\{b^{r+s}a^{n^sx_i+y_j}
\mid \ i\in \{0,1,\dots,k-1\}\},\ j\in \{0,1,\dots,k_2-1\}\} =b^{r+s}a^{n^s\ast A+B}\notag
\end{align}
and $|ST|=|n^s\ast A+B|$.

\medskip

We have proved the following {\bf basic} theorem.

\begin{theorem}\label{t:basic-dilates}
Suppose that 
$$S=b^ra^A\subseteq BS(1,n),\ T=b^sa^B\subseteq BS(1,n)$$
where $r,s \in \N$ and $A,B$ are finite subsets of $\Z$. Then  
$$ST=b^{r+s}a^{n^s\ast A+B}$$
and 
$$|ST|=|n^s\ast A+B|.$$
In particular,
$$S^2=b^{2r}a^{n^r\ast A+A}$$
and
$$|S^2|=|n^r\ast A+A|.$$
\end{theorem}

This result will serve us as the major means for investigating
$|ST|$, and in particular $|S^2|$,
using  information about sizes of sums of dilates.  

\medskip

Skipping to our third topic (c), dealing with inverse problems in 
Baumslag-Solitar groups, it follows from
Theorem \reft{basic-dilates} and from the results 
mentioned in topic (a), that, using the previous notation, the following
statements hold:

{\it {\bf (C).} If $S=ba^A\subseteq BS(1,2)$, then $|S^2|=|2\ast A+A|$. Hence $|S^2|\ge 3|S|-2$
and if $|S^2|= 3|S|-2$, then $A$ is an arithmetic progression.} (see Theorem \reft{basic-group-2}(a), 
Section 2)
 
{\it {\bf (D).} If $S=ba^A\subseteq BS(1,2)$ and $|S^2|< 4|S|-4$, then 
$A$ is a {\bf subset} of an arithmetic progression of size $\leq 2|S|-3$.} (see Theorem
\reft{inverse-one-coset}, Section 4)

{\it {\bf (E).} If $S=ba^A\subseteq BS(1,r)$ with $r\geq 3$, then $|S^2|\ge 4|S|-4$.}
(see Corollary \reft{BS(1,r)}, Section 5) 

{\it {\bf (F).} If $S=b^ma^A\subseteq BS(1,2)$ with $m\ge 2$ an integer, then $|S^2|\ge 4|S|-4$.} 
(see Corollary \reft{4k-4forS}, Section 5)

For more results concerning $S^2$, when $S=ba^A\subseteq BS(1,n)$, see Section 2.
 
\medskip

Conditions of the type $|S^2|< 4|S|-4$ are called {\it small doubling property}.
Our final and main result deals with {\bf arbitary} finite non-abelian 
subsets $S$ of the 
{\it monoid} $BS^+(1,2)$, satisfying the small doubling property $|S^2|<3.5|S|-4$. 
This monoid is defined as follows:
$$BS^{+}(1,2)=\{g=b^ma^x \in BS(1,2): m,x \in \Z,\ m \ge 0\}$$
and it is a subset of $ BS(1,2)$, which is closed with respect to multiplication.

We proved the following
{\bf general} result concerning subsets of $BS^+(1,2)$ (see  Theorem \reft{inverse-general}
in Section 6). 

{\it {\bf (G).} If $S$ is a finite non-abelian subset of $BS^+(1,2)$ satisfying
$$|S^2|<3.5|S| -4,$$
then (i) $|S^2|\ge 3|S|-2$, (ii) $S=ba^A$ for some finite subset $A$ of $\Z$,  
which is {\bf contained} in an arithmetic progression of size $< 1.5|S|-2$
and (iii) $|S^2|=3|S|-2$ implies that $A$ {\bf is} an arithmetic progression of length $|S|$.}

\medskip

Our paper is a pilot study in the following more general direction.
Let $G$ be an infinite non-abelian group of certain type and let 
$S$ denote a finite {\bf non-abelian} subset (i.e. $\langle S\rangle$ is non-abelian) of $G$
of order $k$ ({k}-subset in short). It is natural to ask the following questions:

{\bf Q.1.} Find $m_G(k)$, the minimal possible value of $|S^2|$ for non-abelian
$k$-subsets $S$ of $G$. 

{\bf Q.2.}  What can we say about the detailed structure of {\it extremal $k$-subsets}
of $G$, i.e. finite non-abelian subsets $S$  of $G$ of size $k$, 
satisfying $$|S^2|=m_G(k)?$$

{\bf Q.3.} More generally, what can we say about the detailed structure of 
non-abelian $k$-subsets 
$S$ of $G$, satisfying some {\it small doubling property}, say,
$$m_G(k)\leq |S^2|<c_0k+d_0,$$
where $c_0$ is a small constant greater than $1$ and $d_0$ is some small constant.

\medskip

As mentioned above , we tried to answer these questions in the case of $G=BS(1,n)$
and in particular for $G=BS(1,2)$. 
We hope that our work will lead to similar studies for
other classes of non-abelian groups.

\medskip

This paper is a contribution to the current programme of extending the Freiman-type
theory, concerning the structure of subsets of $\Z$ with the small doubling property,
to such subsets of non-abelian groups (see, for example, \cite{b:gaf}, \cite{b:green} and \cite{b:tct}). 

\medskip

In this paper we  use the following notation. 
We write $[m,n]=[x\in\Z \mid m\le x\le n\}$.
The
{\it algebraic sum\/} of two finite subsets $A$ and $B$ of $\Z$ will be
denoted by 
$$A+B=\{a+b\mid\ a\in A,\ b\in B\}.$$
In particular, if $b\in \Z$, then $A+b=\{a+b:\ a\in A\}.$
The sum $2A=A+A$ is called the
{\it sumset\/} of $A$.
Throughout this paper we shall use the well known inequality
$$|A+B|\geq |A|+|B|-1.$$ 

Let $A = \{a_0<a_1<...<a_{k-1}\}$
be a finite increasing set of $k$ integers. 
By the {\it length \/} $\ell(A)$ of $A$ we mean the
difference 
$$\ell(A)= \max(A) - \min(A)=a_{k-1}-a_0$$ 
between its maximal and
minimal elements and
$$h_A=\ell(A)+1-|A|$$ 
denotes the number of
{\it holes} in $A$,  that is $h_A=|~[a_0,a_{k-1}]\setminus A~|.$
Finally, if $k\ge 2$, then we denote     
$$d(A)=g.c.d.(a_1-a_0,a_2-a_0,...,a_{k-1}-a_0).$$

We shall use several times the following result of Lev-Smelianski and Stanchescu:

{\bf Theorem LSS.} {\it  Let $A$ and $B$ be
finite subsets of $\N$ such that $0 \in A \cap B$.
Define

$$
\delta_{A,B} =
\begin{cases}
1, & \text{if } \ell(A) = \ell(B) , \\
0, & \text{if } \ell(A) \neq \ell(B) .
\end{cases}
$$
Then the following statements hold:
\begin{itemize}
\item[(i)] If $\ell(A) =\max (\ell(A), \ell(B))  \ge |A|+|B|-1 - \delta_{A,B}$
and $d(A)=1$, then $$|A+B|\ge |A|+2|B|-2-\delta_{A,B}.$$
\item[(ii)] If $\max (\ell(A), \ell(B))  \le |A|+|B|-2 - \delta_{A,B}$, then
$$|A+B| \ge (|A|+|B|-1)+\max (h_A,h_B)=\max (\ell(A)+|B|,\ell(B)+|A|).$$
\end{itemize}
}

\begin{proof}
Assertion (i) is Theorem 2(ii) from \cite{b:lev}. Assertion (ii) is 
Theorem 4 from \cite{b:sta}.
\end{proof}

\section{Extremal sets contained in one coset of $BS(1,n)$}

In this section we consider finite subsets $S$ of
$$G=BS(1,n)=\langle a,b\mid ab=ba^n\rangle$$ 
which are contained in the coset $b \langle a \rangle$ 
of $\langle a \rangle$ in $G$. In other words, if $|S|=k$, then
$$ S=b\{a^{x_0},a^{x_1},\dots,a^{x_{k-1}}\}=ba^A,$$
where $A=\{x_0, x_1,\dots, x_{k-1}\}\subseteq \Z$. 

In view of Theorem \reft{basic-dilates}, questions  Q.1 and Q.2 
concerning such $S$ belong to the {\it additive number theory}: 
find a tight lower bound for the size of the Minkowski sum $n\ast A+A$ 
and describe the structure of  extremal sets $A$.

For $n=2$ and $n=3$,
the answer to questions  Q.1 and Q.2
are known. Using  Theorems  1.1
and 1.2 in \cite{b:cill-2} and Theorem \reft{basic-dilates}, 
we get the following  
group-theoretical results:

\begin{theorem}\label{t:basic-group-2}
Let $A\subseteq \Z$ be a finite set of integers.
Then the following statements hold.
\begin{itemize}
\item[(a)] If $S=ba^A \subseteq BS(1,2)$, then
$|S^2| \ge 3|S|-2$.
Moreover, equality
holds  if and only if $A$ is an arithmetic progression.
\item[(b)]If $S=ba^A  \subseteq BS(1,3)$, then 
$$|S^2| \ge 4|S|-4.$$
Moreover, equality 
holds  if and only if either one of the following holds:
$$A=\{0,1,3\},\quad A=\{0,1,4\},\quad  
A = 3\ast \{0,...,n\} \cup (3 \ast \{0,...,n\}+1)$$
or $A$ is an affine transform of one of these sets.
\end{itemize}
\end{theorem}

\begin{proof}

(a) It follows from Theorems  1.1 in \cite{b:cill-2} that
$|A+2\ast A|\ge 3|A|-2$ and $|A+2\ast A|= 3|A|-2$ if and only if
$A$ is an arithmetic progression. Since $|S^2|=|A+2\ast A|$ by Theorem \reft{basic-dilates}, 
we get (a). 

(b) It follows from  Theorem  1.2 in \cite{b:cill-2} that 
$|A+3\ast A| \ge 4|A|-4$ and  $|A+3\ast A|=4|A|-4$ if and only if
either one of the following holds:
$$A=\{0,1,3\},\quad A=\{0,1,4\},\quad
A = 3\ast \{0,...,n\} \cup (3 \ast \{0,...,n\}+1)$$
or $A$ is an affine transform of one of these sets. 
Since $|S^2|=|A+3\ast A|$ by Theorem \reft{basic-dilates}, we get (b).
\end{proof}

\medskip

For $n\geq 4$, Theorem \reft{basic-dilates} and known results
concerning sums of dilates yield the following partial results.

\begin{theorem}\label{t:basic-group-n}
Let $A\subseteq \Z$ be a finite set of integers and let
$S=ba^A$ be a subset of $BS(1,n)$. Then: 
\begin{itemize}
\item[(a)] If $S \subseteq BS(1,4)$ and $|S| \ge 5$, then $|S^2| \ge 5|S|-6$.
\item[(b)] If $S \subseteq BS(1,n)$ , then $|S^2| \ge (n+1)|S|-o(|S|)$.
\item[(c)] If $p$ is an odd prime number, $S \subseteq BS(1,p)$ and 
$|S| \ge 3(p-1)^2(p-1)!$, then 
$$|S^2| \ge (p+1)|S|-\lceil k(k+2)/4\rceil. $$
Moreover, equality holds if and only if 
$A = p\ast \{0,...,m\} + \{0,...,\frac{p-1}{2}\}$ for some $m$.

\end{itemize}
\end{theorem}

\begin{proof}
Using Theorem \reft{basic-dilates}, we get $|S^2| =|n \ast A +A|.$ 

Inequality (a) follows from $|S^2|=|4\ast A+A|$ and Theorem 3 in \cite{b:shan}.

Inequality (b) follows from $|S^2|=|n\ast A+A|$ and Theorem 1.2 in \cite{b:bukh}.

Assertion (c) follows from $|S^2|=|p\ast A+A|$ and Corollary  1.3 in \cite{b:cill-1}.

\end{proof}

\section{An extended inverse result for $|A+2\ast A|$.}
In this section we extend  Theorem 1.1 in \cite{b:cill-2},
which states that $|A+2\ast A|\geq 3|A|-2$ for any finite subset $A$ of $\Z$
and $|A+2\ast A|= 3|A|-2$ implies that $A$ is an arithmetic progression.
In Theorem \reft{inverse-integer} below, 
we prove the following  {\it extended inverse result} in the additive number theory: 
if $A$ is a finite subset of $\Z$ of size $|A|\ge 3 $ 
satisfying  $|A+2\ast A|< 4|A|-4$, 
then $A$ is {\bf contained} in an arithmetic
progression of size $2|A|-3$ at most. This  
result will be used in the next section.

\begin{theorem}\label{t:inverse-integer}
Let  $A=\{a_0<a_1<a_2<\dotsb<a_{k-1}\}\subset \Z$ be a finite set of
integers of size  $k=|A|\ge 1$. Then the following statements hold.
\begin{itemize}
\item[(a)] If $1\le k\le 2$, then $|A+2\ast A|=3k-2$ and
$A$ {\bf is} an arithmetic progression of size $k$.
\item[(b)] If $k\ge 3$, assume that
\begin{align}\label{e:4k-4}
|A+2\ast A|=(3k-2)+h
<4k-4.
\end{align}
Then   
$$h \ge 0,\quad |A+2\ast A| \ge 3k-2$$ 
and the set
$A$ is a {\bf subset} of an arithmetic progression
$$P=\{a_0, a_0+d, a_0+2d,\dots, a_0+(l-1)d\}$$
of size $|P|$ bounded by  
\begin{align}\label{e:BOUND}
|P| \le k+h = |A+2\ast A| -2k+2 \le 2k-3.
\end{align}
\item[(c)] If $k\ge 1$ and $|A+2\ast A|=3k-2$, 
then  $A$ {\bf is} an arithmetic progression
$$A=\{a_0, a_0+d, a_0+2d,\dots, a_0+(k-1)d\}.$$
\end{itemize}
\end{theorem}

\begin{proof}
{\bf (a)} If $k=1$, then $|A+2\ast A|=1=3k-2$
and $A$ is an arithmetic progression of size  $k$.
If $k=2$ and $A=\{a<b\}$, then 
$$A+2\ast A=\{3a,\ a+2b,\ b+2a,\ 3b\}.$$
Since $a\neq b$, it
follows that $|A+2\ast A|=4=3k-2$ and
$A$ is an arithmetic progression of size $k$. The proof of (a) 
is complete. 

{\bf (b)} We assume now that $k\ge 3$ and (\ref{e:4k-4}) holds.
Suppose, first, that $A$ is {\it normal}, i.e.
\begin{align}\label{e:NORMAL_SET}
\min(A)=a_0=0 ~~{\rm and} ~~d=d(A)=gcd(A)
=1.
\end{align}
Thus $\ell(A)=a_{k-1}$.

We split the set $A$ into a disjoint union 
$$A=A_0 \cup A_1,$$
where $A_0 \subseteq 2\Z$ and  $A_1 \subseteq 2\Z +1.$
Since $0=a_0\in A_0$ and $d(A)=1$, it follows that
$A_0 \neq \emptyset$ and $A_1 \neq \emptyset$. 
Therefore
$$m=|A_0| \ge 1, ~~n=|A_1| \ge 1~~{\rm and} ~~k=m+n.$$

We denote
$$A_0=\{0=2x_0 < 2x_1 <...< 2x_{m-1}\}, ~~A_0^*=\frac{1}{2}A_0=\{0<x_1<...<x_{m-1}\},$$
$$A_1=\{2y_0+1 < 2y_1 +1<...< 2y_{n-1}+1\},$$
and
$$A_1^*=\frac{1}{2}(A_1-1)-y_0=\{0<y_1-y_0<y_2-y_0<...<y_{n-1}-y_0\}.$$
Thus
$$\ell(A_0^*)=x_{m-1}<a_{k-1}=\ell(A)~~{\rm and~~also}~~\ell(A_1^*)=
y_{n-1}-y_0<a_{k-1}=\ell(A).$$

The set $A+2\ast A$ is the union of two disjoint subsets
$A_0+2\ast A \subseteq 2\Z$ and $A_1+2\ast A \subseteq 2\Z+1$ and therefore
\begin{align}\label{e:split}
|A+2\ast A|=|A_0+2\ast A|+|A_1+2\ast A|=|A_0^*+A|+|A_1^*+A|.
\end{align}
 
We continue our proof with two claims.

{\bf Claim 1:}
\begin{align}\label{e:length}
\ell(A)
\le k+\max(m,n)-2 \le 2k-3.
\end{align}

For the proof of Claim 1 we shall use Theorem LSS (i).
Since $\ell(A)>\ell(A_0^*),\ell(A_1^*)$, we have
$\delta_{A,A_0^*}=\delta_{A,A_1^*}=0$.

Suppose, first, that $m \le n$. If the claim is false,
then 
$$\ell(A)\ge  k+n-1=|A|+|A_1^*|-1 \ge k+m-1=|A|+|A_0^*|-1$$
and since $d(A)=1$, Theorem LSS (i) yields the following inequalities: 
\begin{align}\label{e:k+2l-2}
|A_0^*+A|  \ge k+2|A_0^*|-2=k+2m-2 ~~{\rm and } ~~ 
|A_1^*+A| & \ge k+2|A_1^*|-2=k+2n-2.
\end{align}
Using (\ref{e:split}) and (\ref{e:k+2l-2}), we get that $|A+2\ast A| \ge 4k-4,$
which contradicts our hypothesis (\ref{e:4k-4}).

Similarly, if $n \le m$ and 
$$\ell(A)\ge  k+m-1\ge k+n-1,$$
then $d(A)=1$ and Theorem LSS (i) imply again the inequalities (\ref{e:k+2l-2}),
which together with (\ref{e:split}) yield $|A+2\ast A| \ge 4k-4,$
a contradiction.

Hence $\ell(A)\le k+\max(m,n)-2 $. Since $k=m+n$ and $m,n\ge 1$,
it follows that $max(m,n)\le k-1 $ and hence  $\ell(A)\le k+\max(m,n)-2\le 2k-3$.
The proof of Claim 1 is complete.

Next we state and prove Claim 2.

{\bf Claim 2:}
\begin{align}\label{e:holes}
|A+2\ast A|
\ge (3k-2)+h_A.
\end{align}

Recall that $h_A=\ell(A)+1-|A|$. For the proof of Claim 2 
we shall use Claim 1 and  Theorem LSS(ii).
We distinguish between two cases.

{\bf Case 1:} Suppose that $m \le n$ and hence, by (\ref{e:length}),
$\ell(A) \le k+n-2. $ 

Thus it follows by Theorem LSS(ii) that 
$$|A_1^*+A| \ge (n+k-1)+h_A$$
and therefore
\begin{align}
|A+2\ast A|&=|A_0^*+A|+|A_1^*+A| \nonumber \\
&\ge (|A_0^*|+|A|-1)+|A_1^*+A| \ge (m+k-1)+ (n+k-1)+h_A\nonumber \\
&=(3k-2)+h_A.\nonumber
\end{align}

{\bf Case 2:} Suppose that $n < m$ and hence, by (\ref{e:length}), 
$\ell(A) \le k+m-2. $ 

Thus it follows by Theorem LSS(ii) that 
$$|A_0^*+A| \ge (m+k-1)+h_A$$
and therefore
$$
|A+2\ast A|=|A_0^*+A|+|A_1^*+A| \ge (m+k-1)+h_A+ (n+k-1)=(3k-2)+h_A.
$$

In both cases we obtain that $h_A$, the total number of holes in 
the normal set $A$, satisfies
$$0 \le h_A  \le |A+2\ast A| -(3k-2)= h \le k-3.$$
Hence 
$$h\geq  h_A \geq 0\quad \text{and}\quad |A+2\ast A| \ge (3k-2).$$
Moreover, the set $A$ is contained in the arithmetic progression 
$$P=\{a_0, a_0+1, a_0+2,..., a_{k-1}\}=\{0,1,2,...,a_{k-1}\}$$
of size
\begin{align}\label{e:partial-conclusion}
a_{k-1}+1=k+h_A \le k+h \le 2k-3.
\end{align}

It follows that Theorem \reft{inverse-integer} (b) holds for  {\it normal} sets $A$ 
satisfying (\ref{e:NORMAL_SET})
and (\ref{e:4k-4}). 

Let now $A$ be an {\it arbitrary} finite set of $k=|A| \ge 3$ integers
satisfying the inequality (\ref{e:4k-4}).
We define 
$$
B=\frac{1}{d(A)}(A-a_0)=\{\frac{1}{d(A)}(x-a_0): x \in A\}.
$$
Note that $|B|=|A|=k,\ \min(B)=0,\ d(B)=1$ and 
$$
|B+2\ast B|=|A+2 \ast A|=(3k-2)+h < 4k-4.
$$
Therefore $B$ is a normal set satisfying inequality (\ref{e:4k-4}) 
of Theorem  \reft{inverse-integer} and as shown above
$$0 \le h_B \le |B+2\ast B| -(3k-2) = |A+2\ast A| -(3k-2)= h \le k-3.$$
Hence also in the general case we get
$$h\geq 0\quad \text{and}\quad |A+2\ast A| \ge (3k-2).$$
Moreover, it follows from (\ref{e:partial-conclusion}) 
applied to $B$ that $B$ is contained in the arithmetic progression
$$P=\{0,1,2,...,b_{k-1}\}$$
with
$$
b_{k-1}=\max(B) \le k+h-1\le 2k-4.
$$
Thus
$A=d(A)B+a_0$ is contained in an arithmetic progression 
$$\{a_0, a_0+d, a_0+2d, ..., a_0+(k+h-1)d\}$$
of size $k+h\leq 2k-3$, where $d$ denotes $d(A)$.
The proof of (b) is complete.  

{\bf (c)} If $1\le k\le 2$, then our claim follows from (a).
So suppose that $k\ge 3$. Then $h=0$ and by (\ref{e:BOUND}) in (b),
$A$ is a subset of an arithmetic progression of size $k$ at most.
But $A$ is a set of size $k$, so $A$ is equal to the arithmetic progression.
The proof of (c), and hence also of 
Theorem \reft{inverse-integer}, is now complete.
 
\end{proof}

\section{An extended inverse result for subsets of $b\langle a\rangle$ in $BS(1,2)$.}
\par
In this section we shall apply Theorem \reft{inverse-integer} in order
to obtain an {\it extended inverse result} in group theory.

Recall that $BS(1,2)=\langle a,b\mid ab=ba^2\rangle$.
In Theorem \reft{basic-group-2} we obtained the following
inverse group-theoretical result:  

{\it If $A\subseteq \Z$ is a finite set of integers
and $S=ba^A \subset BS(1,2)$, then
$$|S^2| \ge 3|S|-2.$$
Moreover, equality
holds  if and only if $A$ is an arithmetic progression.}

Theorem \reft{inverse-integer}, together with Theorem \reft{basic-dilates},
allow us to solve the corresponding
extended  inverse group-theoretical problem.
\begin{theorem}\label{t:inverse-one-coset}
Let $A\subseteq \Z$ be a finite set of integers of size $k=|A| \ge 1$. 
If $S=ba^A$  is a finite subset of the group $BS(1,2)$,
then $|S|=k$ and
\begin{align}\label{e:LB}
|S^2| \ge 3k-2.
\end{align}
Moreover, if $k \ge 3$ and 
\begin{align}\label{e:sdp}
| S^2 | = (3 k-2)+h < 4 | S |-4,
\end{align}
then $h \ge 0$ and  $S$ is a subset of a geometric progression
$$
S \subseteq \{ba^u,ba^{u+d}, ba^{u+2d},...,ba^{u+(k+h-1)d}\}
$$
of size $k+h\leq 2k-3$, where $u=\min(A)$ and $d=d(A)$.

Furthermore, if either $1\le k\le 2$ or $k\ge 3$ and $h=0$,
then $S$ {\bf is} the geometric progression
$$
S =\{ba^u,ba^{u+d}, ba^{u+2d},...,ba^{u+(k-1)d}\}.
$$

\end{theorem}

\begin{proof}
Clearly $|S|=|A|=k$ and by
Theorem \reft{basic-dilates},  $|S^2|=|2\ast A+A|$. 
Hence it follows by Theorem \reft{inverse-integer} that $|S^2| \ge 3k-2$,
proving \refe{LB}.

If $k\ge 3$, then \refe{sdp} implies, again by Theorem \reft{basic-dilates},
that
$$|A+2\ast A|=(3k-2)+h<4k-4.$$
Hence it follows by  Theorem \reft{inverse-integer}, that
$h\ge 0$ and $A$ is a subset of an arithmetic progression
$$P=\{u,u+d,u+2d,\dots,u+(k+h-1)d\}$$
of size $k+h\le 2k-3$, where $u=\min(A)$ and $d=d(A)$. 
Hence
$$
S \subseteq \{ba^u,ba^{u+d}, ba^{u+2d},...,ba^{u+(k+h-1)d}\}.
$$

Finally, if either $1\le k\le 2$ or $k\ge 3$ and $h=0$,
then, by Theorem \reft{inverse-integer}, $A$ {\bf is} an arithmetic progression
and hence $S$ {\bf is} the required geometric progression.

\end{proof}

\section{A new lower bound for $|A+ r \ast A|$ and applications.}

In this section we obtain a new tight lower bound  for $|A+ r \ast A|$,
provided that $r\geq 3$. 

\begin{theorem}\label{t:A+kA}
Let $A=\{a_0<a_1<a_2<\dotsb<a_{k-1}\}\subset \Z$ be a finite set of integers
of size $ |A|=k \ge 1$. Then
for every integer $r \ge 3$ we  have 
\begin{align}\label{e:A+kA}
|A+ r \ast A| \ge \max(4k-4,1)\ge 3k-2.
\end{align}
\end{theorem}

{\bf Remark.} If $r=3$, then Theorem \reft{A+kA} follows from Theorem 1.2
in  \cite{b:cill-2}. If $r \ge 4$, then
the results of \cite{b:bukh} and \cite{b:cill-1} are asymptotically stronger than 
(\ref{e:A+kA}), but we need a lower bound valid for {\it every} $k$. Our
proof is independent of \cite{b:cill-2}. 
\begin{proof}
If $k=1$, then $|A+ r \ast A|=1=\max(4k-4,1)=3k-2$ and the theorem holds.

If $k=2$, then $A=\{a<b\}$ and $r>1$ implies that $a+rb\neq b+ra$. Hence
$$|A+ r \ast A|=|\{a,b\} + \{ra, rb \}| =|\{(r+1)a,b+ra,a+rb,(r+1)b\}|=4=4k-4=3k-2,$$
so the theorem holds also for
$k=2$. Therefore we shall assume,
from now on, that $k\ge 3$. Thus, since $k>1$, we need only to prove that 
$$|A+ r \ast A| \ge 4k-4.$$

We assume first that $A$ is {\it normal}, i.e.
\begin{align}\label{e:NORMAL_SET_2}
\min(A)=a_0=0\qquad {\rm and}\qquad d=d(A)=gcd(A)
=1.
\end{align}

We split the set $A$ into a 
disjoint union of $s$ non-empty subsets, each of which being contained
in a {\bf distinct} residue class modulo $r$:
$$A=A_1 \cup A_2 \cup ... \cup A_s,$$
where 
$$A_i \subseteq x_i+r\Z,\quad |A_i| \ge 1\quad {\rm and} \quad 
x_i=\min A_i.$$
Note that $k\geq 3$, $d(A)=1$ and $\min(A)=a_0=0$, so $s \ge 2$.

We clearly have 
$$|A+ r \ast A|=\sum _{i=1}^{s} |A_i+r \ast A| \ge\sum _{i=1}^{s} (|A_i|+|A|-1)
=|A|+s(|A|-1). $$

If $s \ge 3$, then  we get $|A+ r \ast A| \ge 4|A|-3$
and Theorem  \reft{A+kA} follows.

Hence we may assume that $s=2$ and  $A=A_1 \cup A_2$, where $A_1$ and $A_2$  
are 
non-empty subsets of $A$ contained in disjoint 
residue classes modulo $r$. Let 
$$k_1=|A_1|\quad {\rm and}\quad   k_2=|A_2|.$$
Then $k=k_1+k_2$ and we may assume, without loss of generality, that
$$k_1\ge k_2.$$
Hence $2k_1\ge k$ and $k_1\ge 2$.

Recall that if $S$ is a finite subset of $\Z$, then $\ell(S)$, 
the length of $S$, is defined by 
$\ell(S)=\max(S)-\min(S)$. For $i=1,2$ 
we define 
$$A_i^*= \frac{1}{r}(A_i - \min(A_i))=\{\frac{1}{r}(x-x_i): x\in A_i\}.$$
Clearly $|A_i^*|=|A_i|$ and
we have $$|A_i+r \ast A|=|A_i^*+A|.$$ Thus
$$|A+ r \ast A|=|A_1+r \ast A|+|A_2+r \ast A|=|A_1^*+A|+|A_2^*+A|.$$
Note also that 
$$\ell(A_i) \ge r(k_i-1)\quad {\rm and}\quad 
\ell(A_i^*) =\frac{1}{r}\ell(A_i),$$
so
$$k_i-1 \le \ell(A_i^*) =\frac{1}{r} \ell(A_i) \le \ell(A_i) \le \ell(A).$$ 
Moreover,  $\ell(A_i) >\ell(A_i^*)$ if and only if $k_i>1$, so 
$\ell(A_1) >\ell(A_1^*)$ since $k_1 \ge 2$.

Clearly we must have either $k_1=k-1>k_2=1$
or $k_1\ge k_2>1$. We shall examine these two cases separately.

{\bf Case 1:} Suppose that $k_1=k-1 > k_2=1.$  We have $k=k_1+1$ and
$\ell(A)\ge \ell(A_1) >\ell(A_1^*)$. Moreover,
$$\ell(A) \ge  \ell(A_1) \ge r(k_1-1) \ge 3k_1-3.$$ 
We distinguish now between two complementary subcases.

\noindent
(i) Suppose that $\ell(A) \ge k+k_1-1=2k_1$. Then, since
$d(A)=1$, Theorem LSS(i) implies that 
$$|A+A_1^*| \ge k+2k_1-2.$$
(ii) Suppose that $\ell(A) \le k+k_1-2=2k_1-1$. Then, since $k_1 \ge 2$,  
Theorem LSS(ii) implies that
$$
|A+A_1^*| \ge 
\ell(A) +|A_1| \ge 3k_1-3+k_1 =4k_1-3 \ge 3k_1-1=k +2k_1-2.
$$

Thus in both cases we have
$$|A+ r \ast A|=|A_1^*+ A|+|A_2^*+ A| \ge (k+2k_1-2)+ k=4k-4,$$
as required

{\bf Case 2:} Suppose that $k_1 \ge k_2 >1.$ 
Then
$$\ell(A) > \ell(A_1^*),\quad \ell(A) > \ell(A_2^*)$$ and
$$ \ell(A) \ge \ell(A_i) \ge r(k_i-1) \ge 3k_i-3 $$
for $i=1,2$.
We distinguish now between three complementary subcases.

\noindent
(i) Suppose that  $\ell(A) \ge k+k_1-1. $ 
Then also $\ell(A) \ge k+k_2-1 $ and since $d(A)=1$,
Theorem LSS(i)  implies that
$$
|A+A_1^*| \ge k+2k_1-2,\quad |A+A_2^*| \ge k+2k_2-2.$$
Hence
$$|A+ r \ast A|=|A_1^*+ A|+|A_2^*+ A| \ge (k+2k_1-2)+
(k+2k_2-2)=4k_1+4k_2-4=4k-4,$$
as required.

\noindent
(ii) Suppose that $k+k_2-1 \le \ell(A) \le k+k_1-2.$ 
Then $$k_1 \ge k_2 +1$$ and since
$d(A)=1$, 
Theorem LSS(i) and (ii) imply that
$$
|A+A_1^*| \ge \ell(A)+|A_1^*|  \ge 3k_1-3+k_1=4k_1-3 \ \text{and}\ 
|A+A_2^*| \ge k+2k_2-2.$$
Hence
$$|A+ r \ast A|=|A_1^*+ A|+|A_2^*+ A| \ge 5k_1+3k_2-5 \ge 4k_1+4k_2-4=4k-4,$$
as required.

\noindent
(iii) Suppose that $\ell(A) \le k+k_2-2.$ 
Then $3k_1-3 \le \ell(A) \le k_1+2k_2-2$, yielding 
$2k_1 \le 2k_2+1$. Since $k_1\ge k_2$, it follows that $$k_1 = k_2 \ge 2$$ and
$$
3k_i-3 \le r(k_i-1) \le \ell(A_i) \le  \ell(A) \le k+k_2-2=3k_1-2=3k_2-2.
$$
We claim that $\ell(A)=3k_1-2$. Indeed, if $\ell(A)=3k_1-3$, then
$\ell(A_1)=\ell(A_2)=\ell(A)=a_{k-1}$. But $a_{k-1}\notin A_i$ for some $i$
and hence $\ell(A_i)<a_{k-1}$, a contradiction. This proves our claim.

Recall that $\ell(A) > \ell(A_1^*)$ and $\ell(A) > \ell(A_2^*)$.
Since $\ell(A)=3k_1-2=k+(k_i-2)=|A|+|A_i^*|-2$ for $i=1,2$, it follows,
by Theorem LSS(ii),
that
$$|A+ r \ast A|=|A_1^*+ A|+|A_2^*+ A| \ge \ell(A)+k_1+\ell(A)+k_2=2(3k_1-2)+k
= 4k-4,$$
as required. Our proof in Case 2 is complete.

So Theorem \reft{A+kA} holds for  {\it normal} sets
$A$.
Let $A$ be now an {\it arbitrary} finite set of $k=|A| \ge 3$ integers. 
We define
$$
B=\frac{1}{d(A)}(A-a_0)=\{\frac{1}{d(A)}(x-a_0): x \in A\}.
$$
Note that $|B|=|A|=k$, $\min(B)=0$, $d(B)=1$ and
$|A+r \ast A|=|B+r\ast B|.$
For the normal set $B$ we have proved that $|B+r\ast B| \ge 4|B|-4.$
It follows that
$$
|A+r \ast A|=|B+r\ast B|\ge 4|B|-4= 4|A|-4,
$$
as required. The proof of Theorem \reft{A+kA} is complete. 
\end{proof}
\medskip

Theorem \reft{A+kA} yields the following two applications.
Here is the first one.
\begin{corollary}\label{t:BS(1,r)}
Let $S \subseteq BS(1,r)$ be a finite set of size $ k=|S| \ge 1$
and suppose that $r\geq 3$ and $$S=ba^A,$$
where $A \subseteq \Z$
is a finite set of integers.

Then
\begin{align}\label{e:r=3}
|S^2|=|A+ r \ast A|\ge \max(4k-4,1) \ge 3k-2.
\end{align}
\end{corollary}
\begin{proof} By Theorem \reft{basic-dilates}, $|S^2|=|A+ r \ast A|$
and hence, by Theorem \reft{A+kA},  $|S^2|\ge \max(4k-4,1) \ge 3k-2$,
as required.
\end{proof} 

Our next application will be used several
times in the proof of the main Theorem \reft{inverse-general}
in Section 6.  

\begin{corollary}\label{t:4k-4forS}
Let $S \subseteq BS(1,2)$ be a finite set of size $ k=|S| \ge 1$
and suppose that $$S=b^m a^A,$$
where $m\ge 2$ is an integer and  $A \subseteq \Z$
is a finite set of integers.

Then \begin{align}\label{e:formula}
S^2=b^{2m}a^{A+ 2^m \ast A}.
\end{align}
 and
\begin{align}\label{e:lb-one-plus}
|S^2|=|A+ 2^m \ast A|\ge \max(4k-4,1) \ge 3k-2.
\end{align}

\end{corollary}

\begin{proof}
By Theorem \reft{basic-dilates}, $|S^2|=|A+ 2^m \ast A|$.
Since $2^m>3$, it follows by Corollary \reft{BS(1,r)} that 
$|S^2|\ge \max(4k-4,1) \ge 3k-2$,
as required. 
\end{proof}

\section{An extended inverse result for all subsets of $BS^+(1,2)$.}
\par

In Section 5 we proved an extended inverse result for finite subsets
of $BS(1,2)$ which are contained in the coset $ba^{\Z}$.
In this section  
we solve, using a more detailed analysis, a more general 
problem concerning {\bf all} finite  non-abelian subsets $S$ of 
the corresponding monoid 
\begin{align}\label{e:monoid}
 BS^{+}(1,2)=\{g=b^ma^x \in BS(1,2)\mid \ x,m\in \Z,\ m \ge 0\},
 \end{align}
which  satisfy the more restrictive  small doubling property: 
$$|S^2|< 3.5 |S|-4.$$
We proved the following theorem.

\begin{theorem}\label{t:inverse-general}
If $S$ be a finite non-abelian subset of $BS^{+}(1,2)$ of size $|S|=k,$ 
then
\begin{align}\label{e:first-conclusion}
|S^2| \ge 3k-2.
\end{align}
Moreover, if 
\begin{align}\label{e:general-upper-bound}
|S^2|=(3k-2)+h< 3.5 k-4,
\end{align}
then there exists a finite set
of integers $A \subseteq \Z$ such that 
\begin{itemize}
\item[(a)]$S=ba^A$
\item[(b)] The set $A$ is contained in an arithmetic progression of 
size $$k+h<1.5k-2.$$
\end{itemize}
\end{theorem}

Throughout this section we shall use the following {\bf notation}.
$BS^+(1,2)$ is the monoid defined by (\ref{e:monoid}). Every element 
$g \in BS^+(1,2)$ can be represented in a {\it unique way} as a product
$$
g=b^ma^x,
$$
where $m\in \N$ and $x \in \Z.$  It follows that 
for every two distinct natural numbers $m \neq n$, we have 
\begin{align}\label{e:UNIQUE_REPR}
b^ma^{\Z} \cap b^na^{\Z} =\emptyset.
\end{align}

\noindent 
If  $$S \subseteq BS^{+}(1,2)$$  is a  finite subset of $BS^{+}(1,2)$
of size $k=|S|$, we define a set of natural numbers
$$M_S =M(S)\subseteq \N$$
by the following condition:
$m \in M_S$ if and only if there is  an integer $x$ such that $b^ma^x \in S$.
The set $S$ defines $M_S$ in a unique way and 
we will denote it by 
$$M_S=\{m_0 <m_1  <...<  m_t\},$$
where $t \ge 0$ and $m_0\ge 0$.
For every $0 \le i \le t,$ we define  
\begin{align}\label{e:def-partition}
S_i=S \cap b^{m_i}a^{\Z}, ~~ k_i=|S_i|.
\end{align}
Every set $S_i$ is non-empty, lies in only one coset 
of the cyclic subgroup $\langle a \rangle =a^{\Z}$ and there is a finite set of integers
$A_i \subseteq \Z$ such that 
$$S_i=b^{m_i}a^{A_i} \subseteq b^{m_i}a^{\Z}.$$
The set
$S$ can be written as a {\it disjoint union} of $t+1$ sets
\begin{align}\label{e:partition}
S=S_0 \cup S_1 \cup ... \cup S_t,
\end{align}
satisfying
$$
k_i=|S_i|=|A_i|\ge 1.
$$

{\bf Example 1.} Theorem \reft{inverse-general} is {\it optimal} in view of the following example:
$$S=a^{A_0}\cup \{b\}\subset  BS^{+}(1,2),$$ where 
$$ A_0=\{0,1,2,...,k-2\} ~~\text{{\rm and $k$ is even.}}$$
The set $S$ is clearly non-abelian and
$$S^2=a^{A_0}a^{A_0}\cup ba^{A_0} \cup a^{ A_0}b\cup \{b^2\}.$$
Using  $ a^{ A_0}b=ba^{2 \ast A_0},$ we get
$$S^2=
a^{A_0+A_0}\cup (ba^{A_0} \cup ba^{2 \ast A_0})\cup \{b^2\}=
a^{A_0+A_0}\cup ba^{A_0 \cup 2 \ast A_0} \cup \{b^2\}.$$
Since
$$
a^{A_0+A_0} \subseteq a^{\Z},\quad 
~~ba^{A_0 \cup 2 \ast A_0} \subseteq ba^{\Z},\quad
~\{b^2\} \subseteq b^2a^{\Z},
$$
it follows by (\ref{e:UNIQUE_REPR}) that the three components of $S^2$ are disjoint
in pairs and hence
\begin{align}\label{e:lower}
|S^2|=|A_0+A_0|+|A_0 \cup 2 \ast A_0|+1=(2k-3)+(1.5k-2)+1=3.5k-4.
\end{align}

This example shows that if  
$|S^2|\ge 3.5k-4$, then we have to take into account sets that are not 
included in 
only one coset of the cyclic subgroup $\langle a \rangle$ generated by $a.$

\medskip


The proof of Theorem \reft{inverse-general} will follow from Lemmas 1-7 below.

\begin{lemma}\label{t:nonabe}
Let $S \subseteq BS^+(1,2)$ be a finite set of size $ k=|S|$.
Suppose that $t \ge 1$ and there is $0 \le j \le t$ such that $k_j=|S_j|\ge
2$. Then
$S$ generates a non-abelian group. 
\end{lemma}

\begin{proof}
If $j=0$ {\it and} $m_0 =0$, then 
$
k_0=|S_0|=|A_0| \ge 2
$
implies that $S_0 \neq \{1\}$ and $A_0 \neq \{0 \}.$ 
Since $t\ge 1$, it follows that there are three integers $m, x, z$ 
such that $m \ge 1, x \neq 0$, $a^x \in S_0$  and 
$b^ma^z \in S_1$. In this case
$$a^x(b^ma^z)=b^ma^{z+2^mx} \neq  (b^ma^z)a^x=b^ma^{z+x}$$
and therefore $S$ generates a non-abelian group.

It remains to examine the following two cases: 
\begin{itemize}
\item[(i)] $j \ge 1.$
\item[(ii)] $j=0$ and $m_0 \ge 1.$
\end{itemize}

If $j \ge 1$, then $m_j\ge 1$ and 
$k_j=|S_j|=|b^{m_j}a^{A_j}| \ge 2$ 
implies that $|A_j| \ge 2$.
On the other hand, if $j=0$ and $m_0 \ge 1$, then $k_0=|S_0|=|b^{m_0}a^{A_0}| \ge 2$ 
implies that $|A_0| \ge 2$.
In both cases, let $m=m_j$. Then $m\ge 1$ and there are two integers  $x\ne y$ 
such that 
$\{b^ma^x,b^ma^y\} \subseteq S_j$. We conclude that
$$(b^ma^x)(b^ma^y)=b^{2m}a^{y+2^mx} \neq (b^ma^y)(b^ma^x)=b^{2m}a^{x+2^my},$$
since $x\ne y$ and $m\ge 1$. The proof of Lemma \reft{nonabe} is complete.
\end{proof}

We shall examine now the case $t=1$, i.e. we shall study sets $S$ lying in exactly two cosets.  
Note that inequality (\ref{e:lb-case-two}) in the following Lemma \reft{case2}
is tight, in view of Example 1 .

\begin{lemma}\label{t:case2}
Let $S \subseteq BS^{+}(1,2)$ be a finite set of size  $ k=|S| \ge 2.$
Suppose that $S =U \cup V$ with $U=b^{m}a^{M}\neq \emptyset$ and 
$V=b^{n}a^{N}\neq \emptyset ,$
where  $0 \le m<n $ are two integers and $M, N \subseteq \Z.$ 
Then 
\begin{align}\label{e:lb-case-two}
|S^2|\ge 3.5k-4.
\end{align}
\end{lemma}

\begin{proof}
Clearly $k=|M|+|N|$ and
\begin{align}\label{e:two-cosets-i}
S^2=U^2 \cup (UV \cup VU) \cup V^2.
\end{align}

Using  Theorem \reft{basic-dilates} we get 
\begin{align}\label{e:two-cosets-ii}
U^2=&b^{2m}a^{M+ 2^m \ast M},\quad  
~~V^2 =(b^{n}a^{N})(b^{n}a^{N})= 
b^{2n}a^{N+ 2^n \ast N} ,\\
UV =&(b^{m}a^{M})(b^{n}a^{N})=
b^{m+n}a^{N+ 2^n \ast M}, \quad
~~VU  =(b^{n}a^{N})(b^{m}a^{M})=
b^{m+n}a^{M+ 2^m \ast N}.
\end{align}
Since the sets $b^{2m}a^{\Z}$, $b^{m+n}a^{\Z}$ and $b^{2n}a^{\Z}$ are disjoint 
in pairs, it follows that
\begin{align}\label{e:two-cosets-i-card}
|S^2|=|U^2|+ | (UV \cup VU)|+|V^2|.
\end{align}

We shall examine now two complementary  cases.

{\bf Case 1}: $1 \le m<n .$ 

We shall estimate 
$|U^2|$ and $ |V^2|$
using either Theorem \reft{inverse-one-coset} or Corollary \reft{4k-4forS}. 
We have 
$$|U^2| =|M+ 2^m \ast M| \ge 3|M|-2,\quad  |V^2| = |N+ 2^n \ast N| \ge 3|N|-2.$$
Using  (\ref{e:two-cosets-i-card}) and 
$|UV| =|N+ 2^n \ast M| \ge |M|+|N|-1 $
we conclude that
$$|S^2 | \ge |U^2|+|UV| +|V^2| \ge 3|M|-2 +(|M|+|N|-1)+3|N|-2 =4k-5 \ge 3.5k-4,$$
as required.

{\bf Case 2}: $0 = m<n .$  

In this case $S$ is a disjoint union of two non-empty sets:
$$S=U \cup V,~~{\rm where} ~~U=a^{M}, ~~V=b^{n}a^{N} ~~{\rm and } ~~n\ge 1.$$
We have 
\begin{align}\label{e:two-cosets-2}
&U^2 =a^{M+M},\quad  ~~V^2 = b^{2n}a^{N+ 2^n \ast N}, \\
&UV =b^na^{N+ 2^n \ast M},\quad ~~VU  =b^na^{M+N}.
\end{align}
Therefore it follows, either by Theorem \reft{inverse-one-coset} or by 
Corollary
\reft{4k-4forS},
that
\begin{align}\label{e:two-cosets-3}
|U^2|=|M+M|,\quad |V^2|=|N+2^n \ast N| \ge 3|N|-2.
\end{align}
We also clearly have 
\begin{align}\label{e:two-cosets-4}
|UV \cup VU|&=|(N+ 2^n \ast M) \cup (M+N)| = |(M \cup  2^n \ast M)+N|
\nonumber \\
&\ge |(M \cup  2^n \ast M )|+|N|-1 \ge |M|+|N|-1.
\end{align}
Suppose that $|M|=1$. Then it follows from (\ref{e:two-cosets-i-card}),
(\ref{e:two-cosets-4}) and (\ref{e:two-cosets-3}) that
$$|S^2|\ge 1+(1+|N|-1)+(3|N|-2)=4|N|-1\ge 3.5(1+|N|)-4=3.5|S|-4,$$
as required. So we may assume that $|M|\geq 2$.

We shall complete the proof by dealing separately with two complementary 
subcases.
\noindent
Denote $$\ell=\ell(M)=\max(M) -\min(M),\qquad d=d(M)=\gcd\{x-\min(M):x\in M\}$$
and define
$$M^*=\frac{1}{d}(M-\min(M)),\qquad \ell^*=\ell(M^*)=\max(M^*)=\frac ld.$$

{\bf Case 2.1.} Assume that $\ell(M^*) \ge 2|M^*|-2.$

As shown above, we may assume that $|M|\ge 2$. Suppose that $|M|=2$.
Then $M=\{a_0<a_1\}$, which implies that $d(M)=a_1-a_0$ and $M^*=\{0,1\}$.
Thus $\ell(M^*)=1$ and by our assumptions $1=\ell(M^*)\ge 2|M^*|-2=2$, 
a contradiction. Hence we may assume that $|M|\ge 3$, which implies
that $k=|M|+|N|\ge 3+1=4$.

Note that $d(M^*) =1$. By using  
Theorem LSS(i) for equal summands
we get
\begin{align}\label{e:3m-3}
|U^2|=|M+M|=|M^*+M^*|\ge 3|M^*|-3=3|M|-3.
\end{align}
Using 
(\ref{e:two-cosets-i-card}), (\ref{e:3m-3}), (\ref{e:two-cosets-3}) and (\ref{e:two-cosets-4}),
we may conclude  that
$$|S^2 | \ge |U^2|+|UV \cup VU| +|V^2| \ge (3|M|-3) +(|M|+|N|-1)+(3|N|-2) =4k-6.$$
Since $k\ge 4$, it follows that $|S^2 | \ge 3.5k-4$, as required.

{\bf Case 2.2.} Assume that $\ell(M^*) \le 2|M^*|-3.$ 

In this case, 
we use Theorem LSS(ii) for equal summands. Let 
$h_{M^*}=\ell^*+1-|M^*|$ 
be the number of {\it holes } in $M^*$. We get
\begin{align}\label{e:m-plus-m}
|M+M|=|M^*+M^*|\ge 2|M^*|-1+h_{M^*}=|M^*|+\ell^*=|M|+\ell^*.
\end{align}

We shall now estimate the size of $M \cap  2^n \ast M$.
Note that all the common elements of $2^n \ast M$ and $M$  lie in the interval $[\min(M), \max(M)]$
of length $\ell$ and the set $2^n \ast M$ is included in an arithmetic 
progression of difference $2^nd \ge 2d$. Therefore
\begin{align}\label{e:improve-i}
 |M \cap  (2^n \ast M)|\le \frac{\ell}{2d}+1= \frac{\ell^*}{2}+1
\end{align}
and 
\begin{align}\label{e:improve-ii}
 |M \cup  (2^n \ast M) |=|M|+|2^n \ast M |-  |M \cap  2^n \ast M | \ge 2|M|-\frac{\ell^*}{2}-1.
\end{align}
Using 
(\ref{e:two-cosets-i-card}), (\ref{e:two-cosets-3}), (\ref{e:two-cosets-4}), (\ref{e:m-plus-m})
and  (\ref{e:improve-ii})
we conclude that
\begin{align}
|S^2 | &\ge |U^2|+|UV \cup VU| +|V^2| \ge \\
& \ge |M+M| +(|M \cup  2^n \ast M|+|N|-1)+(|N+ 2^n \ast N|) \nonumber\\
& \ge  (|M|+\ell^*) + (2|M|-\frac{\ell^*}{2}-1+|N|-1)+(3|N|-2) \nonumber\\
& =3|M|+4|N|-4+\frac{\ell^*}{2}  \nonumber\\
& \ge 3|M|+4|N|-4 +\frac{|M^*|-1}{2} = 3.5|M|+4|N|-4 .5 \ge 3.5k-4,\nonumber
\end{align}
as required.
\end{proof}

In Lemmas 3,4,5,6 we shall obtain tight lower bounds for the cardinality 
of $|S^2|$, assuming that $k_i=|S_i| \ge 2$ 
for at most one $i$, $0 \le i \le t.$ 

\begin{lemma}\label{t:particular-case-I}
Let $S \subseteq BS^{+}(1,2)$ be a finite set of size $ k=|S|$.
Suppose that 
\begin{align}\label{e:partition-of-S-i}
S=S_0 \cup S_1 \cup ... \cup S_t, 
\end{align}
where $t \ge 2.$
If $k_0=|S_0| \ge 2$ and
$k_i=|S_i|=1$
for every $1 \le i \le t,$ then 
\begin{align}\label{e:conclusion-I}
|S^2|\ge 4k-5 > 3.5k-4.
\end{align}
\end{lemma}

{\bf Example 2.} Inequality (\ref{e:conclusion-I}) is tight. 

If $$S = \{1,a \} \cup \{b,b^2,...,b^t\},$$ then $k=t+2$ and 
$$
S^2=\{1,a,a^2 \} \cup  \{b,b^2,...,b^t\} \cup \{ab,ab^2,...,ab^t \} \cup \{ba,b^2a,...,b^ta \}
\cup \{b^2,b^3,...,b^{2t}\}.
$$
Note that equality (\ref{e:r-product}) implies that
$$
\{ab,ab^2,...,ab^t \}=\{ba^2,b^2a^4,...,b^ta^{2^t} \}
$$
and thus 
$$S^2=\{1,a,a^2\} \cup \bigcup_{j=1}^{t} b^j\{1,a,a^{2^j}\} 
\cup \{b^{t+1}, b^{t+2},..., b^{2t}\} .$$
Using (\ref{e:UNIQUE_REPR}), we get $|S^2|=3(t+1)+t=4t+3=4k-5.$
\hfill $\square$

We continue now with the proof of Lemma \reft{particular-case-I}.
\begin{proof}
Clearly $k=k_0+t\ge 2+2=4$.
Let $$A_0=\{y_1<...<y_{k_0}\} \subset \Z$$  
be a finite set of $k_0$ integers that defines the set
$$S_0=b^{m_0}a^{A_0}=\{b^{m_0}a^{y_1}, ..., b^{m_0}a^{y_{k_0}}\} $$
with $k_0\ge 2$, and let
$$S_i=\{b^{m_i}a^{x_i}\}$$
for every $1 \le i \le t.$ Recall our assumption that
$0\leq m_0<m_1<\dots <m_t$.

Note that for every $1 \le i \le t$ we have $m_i>0$,
$$S_0S_i=b^{m_0+m_i}\{   a^{x_i+2^{m_i}y_1},...,a^{x_i+2^{m_i}y_{k_0}}     \},\quad ~~|S_0S_i| =k_0$$
and 
$$S_iS_0=b^{m_i+m_0}\{   a^{y_1+2^{m_0}x_i},...,a^{y_{k_0}+2^{m_0}x_i}     \},\quad  ~~|S_iS_0|=k_0.$$

We claim that 
\begin{align}\label{e:plus-one-I}
|S_0S_i \cup S_iS_0|\ge k_0+1.
\end{align}

Indeed, if $S_0S_i=S_iS_0$, then
$$\{ x_i+2^{m_i}y_1<...<x_i+2^{m_i}y_{k_0} \}   =   \{ y_1+2^{m_0}x_i<...<  y_{k_0}+2^{m_0}x_i \}$$
and thus
$$(2^{m_0}-1)x_i=(2^{m_i}-1)y_1=...=(2^{m_i}-1)y_{k_0},$$
which contradicts $\{y_1<...<y_{k_0}\},$ in view of $m_i \ge 1$
and $k_0 \ge 2.$

Note that
$$
S_0S_i \cup S_iS_0 \subseteq b^{m_0+m_i}a^{\Z},\quad S_iS_t  \subseteq b^{m_i+m_t}a^{\Z},
$$
for every $0 \le i \le t.$ Moreover, $S_0S_0=b^{2m_0}a^{A_0+2^{m_0}\ast
A_0}$, so $|S_0S_0|=|A_0+2^{m_0}\ast A_0|\ge 2|A_0|-1=2k_0-1$. 
It follows by (\ref{e:UNIQUE_REPR}) that 
the sets 
$$
S_0S_0, S_0S_1 \cup S_1S_0,...,S_0S_t \cup S_tS_0,S_1S_t,...,S_tS_{t}
$$
are disjoint and included in $S^2$. 
Using $t \ge 2$, $k_0 \ge 2$, (\ref{e:plus-one-I}) and $k\geq 4$, 
we conclude  that 
\begin{align}\label{e:new-IIII}
|S^2| 
&\ge (|S_0S_0|+|S_0S_1 \cup S_1S_0|+...+|S_0S_t \cup S_tS_0|)+(|S_1S_t|+...+|S_tS_{t}|)\nonumber \\
&\ge (2k_0-1)+(k_0+1)+...+(k_0+1)+(1+...+1)=(2k_0-1)+t(k_0+1)+t\nonumber \\
&=4k_0+(t-2)k_0+2t-1\ge 4k_0+2(t-2)+2t-1=4k_0+4t-5=4k-5 \nonumber \\
&> 3.5k-4,
\end{align}
as required.

\end{proof}

\begin{lemma}\label{t:particular-case-II}
Let $S \subseteq BS^{+}(1,2)$ be a finite set of size $ k=|S|$.
Suppose that
\begin{align}\label{e:partition-of-S-ii}
S=S_0 \cup S_1 \cup ... \cup S_t, 
\end{align}
where $t \ge 2.$ If $k_t=|S_t| \ge 2$ and
$k_i=|S_i|=1$
for every $0 \le i \le t-1,$ then 
\begin{align}\label{e:conclusion-II}
|S^2|\ge 4k-5 > 3.5k-4.
\end{align}
\end{lemma}

{\bf Example 3.} Inequality (\ref{e:conclusion-II}) is tight. 

If $$S = \{1,b,b^2,...,b^{t-1}\} \cup  \{b^t,b^ta \},$$ then $k=t+2$ and 
$$
S^2=\{1,b,b^2,...,b^{2t-1}\} \cup  \{ 1, b,...,b^{t-1}\}b^ta
\cup b^t a\{ 1, b,..., b^{t-1}\} \cup \{b^{2t},b^{2t}a, b^tab^t, b^tab^ta \}.
$$
Note that equality (\ref{e:r-product}) implies that
$$
b^t a\{ 1, b,..., b^{t-1}\}=\{b^ta,b^{t+1}a^2,...,b^{2t-1}a^{2^{t-1}} \}
$$
and 
$$
\{b^{2t},b^{2t}a, b^tab^t, b^tab^ta \}=\{b^{2t},b^{2t}a, b^{2t}a^{2^t}, b^{2t}a^{2^t+1} \}.
$$
Thus 
$$S^2=\{1,b,b^2,...,b^{t-1}\} 
\cup b^t\{1,a\} \cup \bigcup_{j=1}^{t-1}
b^{t+j}\{1,a,a^{2^j}\} \cup b^{2t}\{1,a,a^{2^t}, a^{2^t+1}\}$$
and by (\ref{e:UNIQUE_REPR}), $|S^2|=t+2+3(t-1)+4=4t+3=4k-5.$
\hfill $\square$

We continue now with the proof of Lemma \reft{particular-case-II}.
\begin{proof}
Clearly $k=k_t+t\ge 2+2=4$.
Let $$A_t=\{y_1<...<y_{k_t}\} \subseteq \Z$$  
be a finite set of $k_t\geq 2$ integers, which defines the set
$$S_t=b^{m_t}a^{A_t}=\{b^{m_t}a^{y_1}, ..., b^{m_t}a^{y_{k_t}}\} $$
and let
$$S_i=\{b^{m_i}a^{x_i}\}$$
for every $0 \le i \le t-1.$

Note that for every $0 \le i \le t-1$ we have
$$S_tS_i=b^{m_t+m_i}\{   a^{x_i+2^{m_i}y_1},...,a^{x_i+2^{m_i}y_{k_t}}  \},
\quad ~~|S_tS_i| =k_t$$
and 
$$S_iS_t=b^{m_i+m_t}\{   a^{y_1+2^{m_t}x_i},...,a^{y_{k_t}+2^{m_t}x_i}  \},
\quad  ~~|S_iS_t|=k_t.$$
It follows, like in Lemma \reft{particular-case-I}, that
\begin{align}\label{e:plus-one-II}
|S_tS_i \cup S_iS_t|\ge k_t+1
\end{align}
for $1 \le i \le t-1$.

Note that $|S_tS_t| \ge 3k_t -2$, in view of Corollary \reft{4k-4forS}.
Using $|S_0S_t|= k_t $, $k_t \ge 2$, 
(\ref{e:plus-one-II}) and $k\geq 4$, we conclude, like in 
Lemma \reft{particular-case-I},  that 
\begin{align}\label{e:new-2}
|S^2| &\ge (|S_0S_0|+...+|S_0S_t|)+(|S_1S_t \cup S_tS_1|+...+|S_{t-1}S_t\cup S_tS_{t-1}|+|S_tS_t|)\nonumber \\
&\ge (1+...+1+k_t)+((t-1)(k_t+1)+(3k_t-2))=4k_t+(t-1)k_t+2t-3\nonumber \\
&\ge 4k_t+4t-5=4k-5 >3.5k-4,
\end{align}
as required.
\end{proof}
\begin{lemma}\label{t:quasi-abelian}
Let $S \subseteq BS^{+}(1,2)$ be a finite {\it non-abelian} 
set of size $ k=|S| \ge 2.$
Suppose that 
\begin{align}\label{e:partition-of-S-qa}
S=S_0 \cup S_1 \cup ... \cup S_t
\end{align}
where $|S_i|=1$ for all $i$ and 
$$S_i=\{s_i\},\quad s_i=b^{m_i}a^{x_i}.$$ 
Denote $T=S \setminus \{s_0\}$.
 If the subgroup 
$\langle  T \rangle $ is abelian,
then 
\begin{align}\label{e:quasi-lower-bound}
|S^2|\ge 4k-4.
\end{align}
\end{lemma}

\begin{proof}
Recall that $M_S=\{m_0<m_1<\dots <m_t\}$, where $m_0\geq 0$.
We notice first that $T\ne \emptyset$ since $k\ge 2$.
Moreover, we claim that the sets $T^2$, $s_0T \cup Ts_0$ and $\{s_0^2\}$ 
are disjoint. 
Indeed, we have:

(i) $s_0^2 \notin  T^2,$ since $s_0^2 =(b^{m_0}a^{x_0})^2
=b^{2m_0}a^{x_0+2^{m_0}x_0}$ and 
 $T^2 \subseteq \{ b^m a^x  : ~m \ge 2m_1 \}.$

(ii) $s_0^2 \notin (s_0T \cup Ts_0)$, because $s_0 \notin T$.

(iii) $s_0 \notin \langle  T \rangle $, because $\langle  T \rangle $ 
is abelian and 
$\langle  S \rangle $ is non-abelian. This implies that 
$s_0T \cup Ts_0$ does not intersect the set $T^2$.

Notice also that $|T|=t$ and if $s_i,s_j\in T$, then
$s_is_j=b^{m_i+m_j}a^{2^{m_j}x_i+x_j}$, which implies that
$|T^2|\ge |M_S\setminus \{m_0\}+M_S\setminus \{m_0\}|
\ge 2|M_S\setminus \{m_0\}|-1=2t-1$.

In order to complete the proof of  Lemma \ref{t:quasi-abelian}, 
it suffices to show that the sets  $s_0T$ and  $Ts_0$ 
are disjoint. Indeed, if that is the case, then 
\begin{align}
|S^2|
&\ge |T^2|+|s_0T \cup Ts_0| +|\{s_0^2\}|\nonumber \\
&= |T^2|+|s_0T|+|Ts_0| +|\{s_0^2\}|\nonumber \\
&\ge (2t-1)+t+t+1=4t=4|S|-4,\nonumber
\end{align}
as required.

So suppose, by way of contradiction, that   
\begin{align}\label{e:assum}
s_0T \cap Ts_0 \neq \emptyset.
\end{align}
Note that 
$$s_0T=\{s_0s_1,\dots,s_0s_t\}=\{b^{m_0+m_1}a^{x_1+2^{m_1}x_0},\dots,
b^{m_0+m_t}a^{x_t+2^{m_t}x_0}\},$$
and 
$$Ts_0=\{s_1s_0,\dots,s_ts_0\}=\{b^{m_0+m_1}a^{x_0+2^{m_0}x_1},\dots,
b^{m_0+m_t}a^{x_0+2^{m_0}x_t}\}.$$
Therefore (\ref{e:assum}) implies that there is $1 \le i \le t$ such that 
$$
s_0s_i=b^{m_0+m_i}a^{x_i+2^{m_i}x_0}=s_is_0=b^{m_0+m_i}a^{x_0+2^{m_0}x_i}
$$
and thus 
\begin{align}\label{e:h}
(2^{m_i}-1)x_0=(2^{m_0}-1)x_i.
\end{align}
Choose  an arbitrary $1 \le j \le t$. Since $\langle  T \rangle $ is abelian, it follows that 
$$
s_js_i=b^{m_j+m_i}a^{x_i+2^{m_i}x_j}=s_is_j=b^{m_j+m_i}a^{x_j+2^{m_j}x_i},
$$ 
yielding
$$
(2^{m_i}-1)x_j=(2^{m_j}-1)x_i.
$$
Hence
$$
x_i = \frac{2^{m_i}-1}{2^{m_j}-1}x_j
$$
and from (\ref{e:h}) we get 
$$
(2^{m_i}-1)x_0=(2^{m_0}-1)\frac{2^{m_i}-1}{2^{m_j}-1}x_j.
$$
That means that $(2^{m_j}-1)x_0=(2^{m_0}-1)x_j$ and thus
$$
s_0s_j=b^{m_0+m_j}a^{x_j+2^{m_j}x_0}=b^{m_0+m_j}a^{x_0+2^{m_0}x_j}=s_js_0.
$$
It follows that $s_0$ commutes with every element of $T$, which contradicts 
our assumptions that $\langle  T \rangle $ is abelian and 
$\langle  S \rangle $ is non-abelian. The proof of Lemma \ref{t:quasi-abelian}
is complete.
\end{proof}

\begin{lemma}\label{t:particular-case-III}
Let $S \subseteq BS^{+}(1,2)$ be a finite set of cardinality $ k=|S| \ge 2.$
Suppose that $S$ 
is a disjoint union
\begin{align}\label{e:partition-of-S-iii}
S=V_1 \cup V_2 \cup ... \cup V_t, 
\end{align}
of $t$ subsets $$V_i=\{s_i\}, s_i=b^{m_i}a^{x_i},$$ 
of size $|V_i|=1$.
If $S$ is a non-abelian set and $1 \le m_1 <m_2 <..<m_t$, then 
\begin{align}\label{e:conclusion-III}
|S^2|\ge 4k-4.
\end{align}
\end{lemma}

{\bf Example 4.} Inequality (\ref{e:conclusion-III}) is tight.
 
If $S =  \{b,b^2,...,b^{t-1}\} \cup  \{b^ta\} $, then $k=t$ and 
$S^2$ is the union of four disjoint sets: $$\{b^2, b^3,...,b^{2t-2}\},$$
$$ \{b,b^2,...,b^{t-1}\}b^ta=\{b^{t+1}a,b^{t+2}a,...,b^{2t-1}a\},$$
$$b^ta\{b,b^2,...,b^{t-1}\}=\{b^{t+1}a^2,b^{t+2}a^4,...,b^{2t-1}a^{2^{t-1}}\}$$ and 
$\{b^tab^ta\}=\{b^{2t}a^{2^t+1}\}.$
Therefore $$|S^2|=(2t-3)+(t-1)+(t-1)+1=4k-4.$$
\hfill $\square$

We continue now with the proof of Lemma \ref{t:particular-case-III}.

\begin{proof}
If a set $S$  satisfies all the assumptions of Lemma \ref{t:particular-case-III}, 
then we say that $S$ is an {\it elementary set}.

Clearly $t =k\ge 2$ and we proceed by induction on $t$. If 
$t=2$ , then $S=\{s_1,s_2\}$  and since $s_1s_2\ne s_2s_1$ and $s_1^2\ne s_2^2$,
it follows that 
$|S^2|=4=4|S|-4$, as required.

For the inductive step, let $t \ge 3$ be an integer, and assume that 
Lemma \ref{t:particular-case-III} holds for each elementary set 
$T \subseteq BS^{+}(1,2)$ of size $2\leq |T|\leq t-1$. 
Denote
$$S' = S \setminus \{s_1\}.$$ 
In view of Lemma \ref{t:quasi-abelian},
we may assume that $\langle  S' \rangle $ is non-abelian.   

We shall continue by examining two complementary cases.

{\bf Case 1:} $s_1s_2=s_2s_1$. 

Choose $n \ge 2$ {\it maximal} such that the set
$S^*:=\{s_1,s_2,...,s_n\}$ is abelian.
Note that $n<t$, because $S$ in a non-abelian set,
and $s_{n+1}\notin \langle S^*\rangle $. Moreover,  $s_1s_{n+1}\notin S'^2$,
since otherwise
$s_1s_{n+1}=s_us_v$ for some $2\leq u,v\leq t$  and hence 
$b^{m_1+m_{n+1}}=b^{m_u+m_v}$, implying that 
$m_1<m_u,m_v<m_{n+1}$, whence $1<u,v<n+1$ and   
$s_{n+1}\in \langle S^*\rangle $, a contradiction.
Similarly $s_{n+1}s_1\notin S'^2$. 

We claim that it suffices to show that $s_{n+1}$  does not commute with $s_1$.

Indeed, if $s_1s_{n+1}\ne s_{n+1}s_1$, then 
(\ref{e:conclusion-III}) follows from
$$\{s_1^2, s_1s_2, s_1s_{n+1}, s_{n+1}s_1 \} \subseteq S^2 \setminus S'^2$$
and from the induction hypothesis for $S'$:
$$|S^2| \ge |S'^2|+|\{s_1^2, s_1s_2, s_1s_{n+1}, s_{n+1}s_1 \}| \ge
(4|S'|-4)+4=4|S|-4. $$

We shall complete the proof by showing that if
\begin{align}\label{e:commutative}
s_1s_{n+1}=s_{n+1}s_1,
\end{align}
then
$$
s_js_{n+1}=s_{n+1}s_j,
$$
for every $1 \le j \le n$, which contradicts the maximality of $n$.

Our argument is similar to that used  in the proof of 
Lemma \ref{t:quasi-abelian}. Denote
$m=m_{n+1}$, $x=x_{n+1}$ and
$$
s_{n+1}=b^ma^x.
$$
We first note that (\ref{e:commutative}) implies that

\begin{align}
s_1s_{n+1} &=(b^{m_1}a^{x_1})(b^ma^x)=b^{m_1+m}a^{x+2^{m}x_1}\notag\\
&=s_{n+1}s_1=(b^ma^x)(b^{m_1}a^{x_1})=b^{m+m_1}a^{x_1+2^{m_1}x}\notag
\end{align}
and thus
\begin{align}\label{e:h-new}
(2^{m_1}-1)x=(2^{m}-1)x_1.
\end{align}
Choose  an arbitrary $1 \le j \le n$. Using
$$s_1s_j=s_js_1$$
we get, like in  the proof of Lemma \ref{t:quasi-abelian}, that
$$
x_1 = \frac{2^{m_1}-1}{2^{m_j}-1}x_j.
$$
It follows by (\ref{e:h-new}) that
$$
(2^{m_1}-1)x=(2^{m}-1)\frac{2^{m_1}-1}{2^{m_j}-1}x_j
$$
and since $m_1 \ge 1$, we may conclude that $$(2^{m_j}-1)x=(2^{m}-1)x_j.$$ 
Thus $s_js_{n+1}=s_{n+1}s_j$,
a contradiction.

{\bf Case 2:} $s_1s_2 \neq s_2s_1$. 

We {\it claim} that 
\begin{align}\label{e:three-elements}
{\rm either}\quad s_1s_3 \neq s_2^2\quad {\rm or }\quad s_3s_1 \neq s_2^2.
\end{align}

Indeed, if 
$$
s_1s_3 = s_2^2 \quad {\rm and  } \quad s_3s_1 = s_2^2
$$
then
$$
(b^{m_1}a^{x_1})(b^{m_3}a^{x_3})=(b^{m_2}a^{x_2})^2=(b^{m_3}a^{x_3})(b^{m_1}a^{x_1})
$$
and thus
$$
b^{m_1+m_3}a^{x_3+2^{m_3}x_1}=b^{2m_2}a^{x_2+2^{m_2}x_2}=b^{m_1+m_3}a^{x_1+2^{m_1}x_3}.
$$
It follows that $m_1+m_3=2m_2$ and 
$$x_3 = x_2(2^{m_2}+1)-2^{m_3}x_1,\quad
2^{m_1}x_3=x_2(2^{m_2}+1)-x_1.$$
Thus
$$2^{m_1}x_3=(2^{m_2}+1)x_2-x_1=2^{m_1}(2^{m_2}+1)x_2-2^{m_1+m_3}x_1,
$$
implying that 
$$(2^{m_1}-1)(2^{m_2}+1)x_2=(2^{m_1+m_3}-1)x_1=(2^{2m_2}-1)x_1.
$$
Hence
$$
(2^{m_1}-1)x_2=(2^{m_2}-1)x_1,
$$
and 
$$s_1s_2=b^{m_1+m_2}a^{x_2+2^{m_2}x_1}=b^{m_2+m_1}a^{x_1+2^{m_1}x_2}=s_2s_1,$$ 
a contradiction. The proof of our claim is complete.

Thus
$$|\{ s_1s_3, s_3s_1 \} \setminus \{s_2^2\}|\ge 1.$$
Since
$$\{s_1^2, s_1s_2, s_2s_1\}  \subseteq S^2 \setminus S'^2
$$
and 
$$
\{ s_1s_3, s_3s_1 \} \setminus \{s_2^2\} \subseteq S^2 \setminus S'^2,
$$
it follows by the induction hypothesis for $S'$, that   
$$|S^2| \ge |S'^2|+|\{s_1^2, s_1s_2, s_2s_1, s_1s_3, s_3s_1\} \setminus S'^2| 
\ge (4|S'|-4)+4=4|S|-4. $$

The proof of Lemma \ref{t:particular-case-III} is complete.

\end{proof}

The following lemma is the main step in the proof of Theorem \reft{inverse-general}. 
We use an inductive argument analogous to that used for
the proof of  Lemma 2.2 in  \cite{b:sta-plane} (see also Lemma 3 in \cite{b:sta-dim}).

\begin{lemma}\label{t:case3}
Let $S \subseteq BS^{+}(1,2)$ be a finite set of size $ k=|S| \ge 2.$
Suppose that  
\begin{align}\label{e:partition-of-S}
S=S_0 \cup S_1 \cup ... \cup S_t,
\end{align}
where $ t \ge 1$. If $S$ is a non-abelian set, then 
\begin{align}\label{e:conclusion}
|S^2|\ge 3.5k-4.
\end{align}
\end{lemma}

\begin{proof}

We 
use induction on $t \ge 1$. Observe that when 
$t=1$, Lemma \ref{t:case3} follows from Lemma \ref{t:case2}.

For the inductive step, let $t \ge 2$ be an integer, and assume that 
Lemma \ref{t:case3} holds for any non-abelian finite set 
$T \subseteq BS^{+}(1,2)$ which lies in 
$u$ distinct cosets of $\langle a \rangle = a^{\Z}$, where $2 \le u < t.$

Denote $$S^*=S \setminus S_t,\quad k^*=|S^*|=k-k_t.$$
If $S^*$ generates a non-abelian group, then our inductive 
hypothesis implies that 
$$|(S^*)^2|\ge 3.5k^*-4,$$ and  
it suffices to show that
\begin{align}\label{e:new}
|S_t^2 \cup S_tS_{t-1} \cup S_{t-1}S_t |\ge 3.5k_t,
\end{align}
since inequality (\ref{e:conclusion})  then follows from
\begin{align}\label{e:qed}
|S^2|\ge |(S^*)^2|+|S_t^2 \cup S_tS_{t-1} \cup S_{t-1}S_t |\ge
(3.5k^*-4)+3.5k_t=3.5k-4.
\end{align}

The proof of (\ref{e:new}) will be provided by examining four 
complementary cases.

{\bf Case 1:}  Assume that either $ k_t \ge 2$ or $k_{t-1} \ge 2.$

Recall that $0 \le m_0 < m_1 <...< m_t$ 
and hence $t \ge 2$ implies that $m_t \ge 2$. Thus, 
using either Theorem \reft{inverse-one-coset} or Corollary \ref{t:4k-4forS},
we get 
$$|S_t^2| \ge \max \{4k_t -4, 1\}  \ge 3k_t-2.$$

We shall examine now four subcases.

{\bf i.} If $k_t+2k_{t-1} \ge 6,$  then (\ref{e:new}) is true in view of:
$$|S_t^2 \cup S_tS_{t-1} \cup S_{t-1}S_t |\ge |S_t^2 |+| S_tS_{t-1}|
\ge (3k_t-2)+(k_t+k_{t-1}-1)=4k_t+k_{t-1}-3
\ge 3.5k_t.$$

If there is $0 \le j \le t-1$ such that $k_j=|S_j|\ge 2,$ then
$S^*$ generates a non-abelian group (in view of Lemma \reft{nonabe})
and  we may apply the induction hypothesis. Thus, 
Lemma \ref{t:case3} follows from (\ref{e:new}) and (\ref{e:qed}).

If $k_j=1$ for all $0 \le j \le t-1$, then $k_t \ge 6 - 2k_{t-1}\ge 4$, 
in view of ({\bf i}).
In this case, Lemma \ref{t:case3} follows from Lemma \ref{t:particular-case-II}.

So we may assume that Case ({\bf i}) does not hold and in particular 
$$ 
k_t+2k_{t-1} \le 5.
$$
Hence one of the following cases must hold: ({\bf ii}) $k_t=3, k_{t-1}=1$, 
({\bf iii}) $k_t=2, k_{t-1}=1$ or ({\bf iv}) $k_t=1, k_{t-1}=2.$

{\bf ii.} If $k_t=3$ and $k_{t-1}=1$ , then Corollary \reft{4k-4forS}  
implies $|S_t^2| \ge 4k_t-4$ and therefore inequality (\ref{e:new}) follows from:
$$|S_t^2 \cup S_tS_{t-1} \cup S_{t-1}S_t |\ge |S_t^2 |+| S_tS_{t-1}| 
\ge (4k_t-4)+(k_t+k_{t-1}-1)=5k_t-4>3.5k_t.$$

If there is $0 \le j \le t-1$ such that $k_j=|S_j|\ge 2,$ then
$S^*$ generates a non-abelian group (in view of Lemma \reft{nonabe})
and  we may apply the induction hypothesis. Thus, 
Lemma \ref{t:case3} follows from (\ref{e:new}) and (\ref{e:qed}).

If $k_j=1$ for all $0 \le j \le t-1$, then
Lemma \ref{t:case3} follows from Lemma \ref{t:particular-case-II},
in view of $k_t=3.$

{\bf iii.} If $k_t=2$ and $k_{t-1}=1,$ then we can write
$$S_{t-1}=\{b^ua^x\},\quad S_t = \{b^va^y, b^va^z\},$$
where $1 \le u=m_{t-1}<v=m_t$ and $y<z$ are integers.
Using the identity $a^xb^m=b^ma^{2^m x} $, we get
$$S_{t-1}S_t=b^{u+v}\{a^{2^vx+y},a^{2^vx+z}\}\quad {\rm and}
\quad S_tS_{t-1}=b^{u+v}\{a^{2^uy+x},a^{2^uz+x}\}.$$
Note that $S_{t-1}S_t \neq S_tS_{t-1}$. Indeed, 
if $S_{t-1}S_t=S_tS_{t-1}$, then $2^vx+y=2^uy+x$ and $2^vx+z=2^uz+x$.
Thus $(2^u-1)y=(2^v-1)x=(2^u-1)z$, which contradicts $y<z$, 
in view of $u \ge 1.$
Therefore either Theorem \reft{inverse-one-coset} or Corollary \ref{t:4k-4forS} 
implies that
$$|S_t^2 \cup S_tS_{t-1} \cup S_{t-1}S_t |=
|S_t^2| +  |S_tS_{t-1} \cup S_{t-1}S_t| \ge 
(3k_t-2)+3 = 4+3 =3.5k_t. $$

If there is $0 \le j \le t-1$ such that $k_j=|S_j|\ge 2,$ then
$S^*$ generates a non-abelian group (in view of Lemma \reft{nonabe})
and  we may apply the induction hypothesis. Thus,
Lemma \ref{t:case3} follows from (\ref{e:new}) and (\ref{e:qed}).

If $k_j=1$ for all $0 \le j \le t-1$, then
Lemma \ref{t:case3} follows from Lemma \ref{t:particular-case-II},
in view of $k_t = 2.$

{\bf iv.} If $k_t=1$ and $k_{t-1}=2,$ then we can write
$$S_{t-1} = \{b^ua^y, b^ua^z\}, ~~S_t=\{b^va^x\}, $$
where $1 \le u=m_{t-1}<v=m_t$, $x,y,z$ are integers and $y<z$.
Using the identity $a^xb^m=b^ma^{2^m x} $, we get
$$S_{t-1}S_t=b^{u+v}\{a^{2^vy+x},a^{2^vz+x}\}\quad {\rm and}
\quad S_tS_{t-1}=b^{u+v}\{a^{2^ux+y},a^{2^ux+z}\}.$$
Note that $S_{t-1}S_t \neq S_tS_{t-1}$. Indeed, 
if $S_{t-1}S_t=S_tS_{t-1}$, then $2^vy+x=2^ux+y$ and $2^vz+x=2^ux+z$.
Thus $(2^v-1)y=(2^u-1)x=(2^v-1)z$, which contradicts $y<z$, 
in view of $v \ge 1.$
Therefore, 
$$|S_t^2 \cup S_tS_{t-1} \cup S_{t-1}S_t |=
|S_t^2 |+| S_tS_{t-1} \cup S_{t-1}S_t | \ge 
 1+3 >3.5k_t. $$
Since $k_{t-1}= 2,$  Lemma \reft{nonabe} implies that
$S^*$ generates a non-abelian group 
and  we may apply the induction hypothesis. Thus, 
Lemma \ref{t:case3} follows from (\ref{e:new}) and (\ref{e:qed}).

The proof in {\bf Case 1} is complete. 

{\bf Case 2:}  Assume that $k_t=k_{t-1} =...=k_1=1$ and 
and $k_0\ge 2 $. 

In this case, Lemma \reft{case3} 
follows form Lemma \reft{particular-case-I}. 

{\bf Case 3:}  Assume that $k_t=k_{t-1} =1$ and there is $1 \le j \le t-2$ such that $k_j \ge 2$
and $k_i =1 $ for every $i \in \{j+1,...,t\}.$

Let   

$$S_j=b^{m_j}a^{A_j}=\{b^{m_j}a^{y_1}, ..., b^{m_j}a^{y_{k_j}}\} $$
and let 
$S_i=\{b^{m_i}a^{x_i}\}$
for every $i \in \{j+1,...,t\}.$
Clearly $|S_jS_i |=|S_iS_j|=k_j$ and
using the reasoning in the proof of Lemma \reft{particular-case-I}, 
we get 
\begin{align}\label{e:union}
|S_jS_i \cup S_iS_j|\ge k_j+1
\end{align}
for every $i \in \{j+1,...,t\}.$

Note that $k_j \ge 2 $, so by Lemma \reft{nonabe} the set  
$$S^*_j=S_0 \cup S_1 \cup ... \cup S_j$$
is non-abelian. By applying the inductive hypothesis to $S^*_j$ and in
view of (\ref{e:union}), we
obtain
\begin{align}\label{e:new-II}
|S^2| 
&\ge |S^*_jS^*_j|+\sum_{u=j+1}^{t}|S_jS_u \cup S_uS_j|+\sum_{u=j+1}^{t}|S_uS_t|\\
&\ge (3.5|S^*_j|-4)+(k_j+1)(t-j)+(t-j)\nonumber \\
&=(3.5|S^*_j|-4)+(k_j+2)(t-j)\ge (3.5|S^*_j|-4)+4(t-j)\nonumber \\
&= (3.5|S^*_j|-4)+4(k-|S^*_j|)=4k-4-0.5|S^*_j|>3.5k-4,
\end{align}
as required.

{\bf Case 4:} Assume that $k_i=1$ for every  $0 \le i \le t.$

If the set $S' =S \setminus  \{s_0\}$ is abelian, 
then Lemma \reft{quasi-abelian}
implies that $$|S^2|\ge 4|S|-4 > 3.5k-4,$$
as required. Therefore, we may assume that
$S'=S_1 \cup S_2 \cup ... \cup S_t $
is non-abelian. Since $t\geq 2$, it follows that $k=t+1\geq 3$ and Lemma
\reft{particular-case-III} implies that
\begin{align}
|S'^2|\ge 4|S'|-4.
\end{align}
Moreover,$\{s_0^2,s_0s_1\} \subseteq S^2 \setminus S'^2$.

We distinguish now between two complementary cases.

(a) If $k=|S| \ge 4$, then  
$$
|S^2| \ge |S'^2|+2  \ge 4(k-1)-4+2=4k-6\ge 3.5k-4,
$$ 
as required.
   
(b) If $k=3$, then $S=\{s_0,s_1,s_2\}$, $S'=\{s_1,s_2\}$, $s_1s_2 \neq s_2s_1$, 
$|S'^2|=4$ and
$$
S^2= \{s_0^2,s_0s_1, s_1s_0, s_0s_2, s_2s_0\} \cup S'^2.
$$
We {\it claim} that 
\begin{align}\label{e:three-elements-last}
{\rm either}\quad s_0s_2 \neq s_1^2 \quad {\rm or } \quad s_2s_0 \neq s_1^2.
\end{align}

Indeed, if 
$$
s_0s_2 = s_1^2 \quad {\rm and  } \quad s_2s_0 = s_1^2
$$
then
$$
(b^{m_0}a^{x_0})(b^{m_2}a^{x_2})=(b^{m_1}a^{x_1})^2=(b^{m_2}a^{x_2})(b^{m_0}a^{x_0})
$$
and thus
$$
b^{m_0+m_2}a^{x_2+2^{m_2}x_0}=b^{2m_1}a^{x_1+2^{m_1}x_1}=b^{m_2+m_0}a^{x_0+2^{m_0}x_2}.
$$
It follows that $m_0+m_2=2m_1$ and 
$$2^{m_2}x_0=x_1(2^{m_1}+1)-x_2,\quad x_0=x_1(2^{m_1}+1)-2^{m_0}x_2.$$
Thus 
$$2^{m_2}x_0=x_1(2^{m_1}+1)-x_2=2^{m_2}(2^{m_1}+1)x_1-2^{m_2+m_0}x_2,$$
implying that 
$$
(2^{m_2}-1)(2^{m_1}+1)x_1=(2^{m_0+m_2}-1)x_2=(2^{2m_1}-1)x_2.
$$
Hence
$$
(2^{m_2}-1)x_1=(2^{m_1}-1)x_2
$$
and 
$s_1s_2=s_2s_1$,
a contradiction.

We conclude that
$$
|S^2| \ge  |\{s_0^2,s_0s_1, s_0s_2, s_2s_0\}\setminus S'^2|+|S'^2|\ge 3+4=7 > 3.5|S|-4.
$$

The proof of Lemma \ref{t:case3} is complete.

\end{proof}

{\it Proof of Theorem  \reft{inverse-general}.}

Let $S$ be a finite set satisfying the assumptions of Theorem \reft{inverse-general}.
Inequality $$|S^2| < 3.5k -4$$ and Lemma \reft{case3} imply that 
$$
S=S_0=b^{m_0 }a^{A_0},
$$
where $m_0 \ge 0$. 

The set  $S$ is non-abelian, so $m_0 \ge 1$. If $ m_0 \ge 2 $, then 
Corollary \ref{t:4k-4forS} implies that 
$|S^2 | \ge 4k-4 > 3.5k-4$, which contradicts our hypothesis. Therefore
$$S=S_0=ba^{A}.$$
 
Theorem  \reft{inverse-general} now follows from Theorem \reft{inverse-one-coset}.
 \hfill $\square$

\end{document}